\documentclass[10pt,reqno]{amsart}

\usepackage{amsmath,amsfonts,amsbsy,amsgen,amscd,mathrsfs,amssymb,amsthm}
\usepackage{enumerate,mathtools}

\usepackage[usenames,dvipsnames]{xcolor}
\usepackage[colorlinks=true,citecolor=blue,linkcolor=blue]{hyperref}

\usepackage{tikz}
\usetikzlibrary{shadings}

\newtheorem{thm}{Theorem}[section]
\newtheorem{lem}[thm]{Lemma}

\newtheorem{prop}[thm]{Proposition}

\theoremstyle{definition}

\newtheorem{rem}[thm]{Remark}

\numberwithin{equation}{section} 
\numberwithin{figure}{section}
\numberwithin{table}{section}

\begin{document}

\title{Rademacher Type and Enflo Type Coincide}

\author{Paata Ivanisvili}
\address{(P.I.) 
Department of Mathematics, North Carolina State University, Raleigh, NC
27695, USA; and Department of Mathematics, University of California, 
Irvine, CA 92617, USA}
\email{pivanisv@uci.edu}

\author{Ramon van Handel}
\address{(R.v.H.) Fine Hall 207, Princeton University, Princeton, NJ 
08544, USA}
\email{rvan@princeton.edu}

\author{Alexander Volberg}
\address{(A.V.) Department of Mathematics, Michigan State University, 
East Lansing, MI 48823, USA}
\email{volberg@math.msu.edu}

\begin{abstract}
A nonlinear analogue of the Rademacher type of a Banach space
was introduced in classical work of Enflo. The key feature of Enflo type 
is that its definition uses only the metric structure of the Banach space, 
while the definition of Rademacher type relies on its linear structure.
We prove that Rademacher type and Enflo type coincide, settling a 
long-standing open problem in Banach space theory. The proof is based on
a novel dimension-free analogue of Pisier's inequality on the discrete cube.
\end{abstract}



\maketitle

\thispagestyle{empty}

\section{Introduction and main results}
\label{sec:intro}

Let $(X,\|\cdot\|)$ be a Banach space. We say that $X$ has
\emph{Rademacher type} $p\in[1,2]$ if there exists
$C\in(0,\infty)$ so that for all $n\ge 1$ and 
$x_1,\ldots,x_n\in X$
$$
	\mathbf{E}\Bigg\|\sum_{j=1}^n\varepsilon_jx_j\Bigg\|^p
	\le
	C^p\sum_{j=1}^n \|x_j\|^p.
$$
We denote by $T_p^\mathrm{R}(X)$ the smallest possible constant $C$ in 
this inequality.

A nonlinear notion of type was introduced by Enflo 
\cite{Enf78}: a Banach space has \emph{Enflo 
type} $p$ if there exists $C\in(0,\infty)$ so that
for all $n\ge 1$ and $f:\{-1,1\}^n\to X$
$$
	\mathbf{E}\bigg\|\frac{f(\varepsilon)-f(-\varepsilon)}{2}\bigg\|^p 
	\le
	C^p \sum_{j=1}^n \mathbf{E}\|D_jf(\varepsilon)\|^p,
$$
and we denote by $T_p^\mathrm{E}(X)$ the smallest possible constant $C$ in   
this inequality. Here we define the discrete partial derivatives
on the cube $\{-1,1\}^n$ as
$$
	D_jf(\varepsilon) := 
	\frac{f(\varepsilon_1,\ldots,\varepsilon_j,\ldots,\varepsilon_n)
	-f(\varepsilon_1,\ldots,-\varepsilon_j,\ldots,\varepsilon_n)}{2}.
$$
The key feature of Enflo type is that its definition depends only on the 
metric structure of $X$, that is, it involves only distances 
between two points. This notion therefore extends 
naturally to the setting of general metric spaces. In contrast, the 
definition of Rademacher type relies on the linear structure of $X$.

The study of metric properties of Banach spaces, known as the ``Ribe 
program'', has been of central importance in Banach space theory in 
recent decades 
\cite{Nao12}. Understanding the relationship between Rademacher type and 
Enflo type is a fundamental question in this program. That Enflo type 
$p$ implies Rademacher type $p$ follows immediately by choosing the linear 
function $f(\varepsilon)=\sum_{j=1}^n \varepsilon_jx_j$ in the definition 
of Enflo type. Whether the converse is also true, that is, that Rademacher 
type $p$ implies Enflo type $p$, is a long-standing problem that 
dates back to Enflo's original paper \cite{Enf78} from 1978. Despite a 
number of partial results in this direction 
\cite{Enf69,Enf70,BMW86,Pis86,NS02,MN07,HN13,EN19}, 
the question has remained open.

Here we settle Enflo's question in the affirmative: Rademacher type $p$ is 
equivalent to Enflo type $p$. In other words, Enflo type provides a 
characterization of Rademacher type using only the metric structure of $X$.

\begin{thm}
\label{thm:enflo}
We have
$$
	T_p^\mathrm{R}(X) \le T_p^\mathrm{E}(X) \le
	\frac{\pi}{\sqrt{2}}T_p^{\mathrm{R}}(X)
$$
for every $p\in[1,2]$ and Banach space $X$.
\end{thm}

The key new ingredient in the proof of Theorem \ref{thm:enflo} is a novel 
dimension-free analogue of a classical inequality of Pisier.

\subsection{Pisier's inequality}

Let $p\ge 1$, let $f:\{-1,1\}^n\to X$ and let $\varepsilon,\delta$ be 
independent random vectors that are uniformly distributed on the discrete 
cube $\{-1,1\}^n$. As part of his investigation of metric type, Pisier 
discovered the following class of Sobolev-type inequalities for 
vector-valued functions on the discrete cube:
\begin{equation}
\label{eq:pisierc}
	\mathbf{E}\|f(\varepsilon)-\mathbf{E}f(\varepsilon)\|^p \le
	C^p\, \mathbf{E}\Bigg\|\sum_{j=1}^n \delta_j D_jf(\varepsilon)
	\Bigg\|^p.
\end{equation}
If such an inequality were to hold with a constant $C$ that is independent 
of dimension $n$, then Enflo's problem would be solved: if $X$ has 
Rademacher type $p$, then applying this property to the right-hand side of 
\eqref{eq:pisierc} conditionally on $\varepsilon$ would yield immediately 
the definition of Enflo type $p$. Unfortunately, Pisier was able to prove 
\eqref{eq:pisierc} only with a dimension-dependent constant $C\sim\log n$ 
\cite[Lemma 7.3]{Pis86}, and it was subsequently shown by Talagrand 
\cite[section 6]{Tal93} that this order of growth is optimal: that is, 
there exist Banach spaces $X$ for which the optimal constant in 
Pisier's inequality must grow logarithmically with dimension.

In order to resolve Enflo's problem, however, it is not necessary to 
establish Pisier's inequality for an \emph{arbitrary} Banach space: it 
suffices to show that \eqref{eq:pisierc} holds with a dimension-free 
constant under the additional assumption that $X$ has nontrivial type. For 
this reason, subsequent work has focused on identifying conditions on the 
Banach space $X$ under which \eqref{eq:pisierc} holds with a constant that 
depends only on the geometry of $X$ (but not on $n$). Notably, Naor and 
Schechtman \cite{NS02} proved that \eqref{eq:pisierc} holds with a 
dimension-free constant under the stronger assumption that $X$ is an UMD 
Banach space (see also \cite{HN13,Esk20}). Very recently,
Eskenazis and Naor \cite{EN19} proved that for superreflexive Banach
spaces $X$, the constant in Pisier's inequality can be improved to
$\log^\alpha n$ for some $\alpha<1$.

Beside the inequality \eqref{eq:pisierc}, Pisier also proved \cite[Theorem 
2.2]{Pis86} a more general counterpart of his inequality in Gauss space:
if $f:\mathbb{R}^n\to X$ is 
locally Lipschitz, $G,G'$ are independent standard Gaussian 
vectors in $\mathbb{R}^n$, and $\Phi:X\to\mathbb{R}$ is 
convex and satisfies a mild regularity assumption, then
\begin{equation}
\label{eq:pisierg}
	\mathbf{E}[\Phi(f(G)-\mathbf{E}f(G))] \le
	\mathbf{E}\bigg[\Phi\bigg(\frac{\pi}{2}\sum_{j=1}^n
	G_j' \frac{\partial f}{\partial x_j}(G)\bigg)\bigg].
\end{equation}
One obtains an inequality analogous to \eqref{eq:pisierc} by choosing 
$\Phi(x)=\|x\|^p$. Remarkably, the Gaussian inequality is 
dimension-free for an arbitrary Banach space $X$, in sharp contrast to the 
inequality on the cube. Unfortunately, its proof is very special to the 
Gaussian case: one defines $G(\theta):=G\sin\theta+G'\cos\theta$, and 
notes that $(G(\theta),\frac{d}{d\theta}G(\theta))$ has the same 
distribution as $(G,G')$ for each $\theta$ by rotation-invariance of the 
Gaussian measure. Then \eqref{eq:pisierg} follows by expressing 
$f(G)-f(G')=\int_0^{\pi/2} \frac{d}{d\theta}f(G(\theta))\,d\theta$ and 
applying Jensen's inequality. If one attempts to repeat this idea on the 
discrete cube, the absence of rotational symmetry makes the argument 
inherently inefficient, and one cannot do better than \eqref{eq:pisierc} 
with constant $C\sim\log n$.

Despite the apparent obstructions, we will prove in this paper a 
completely general dimension-free analogue of \eqref{eq:pisierg} on the 
discrete cube. The existence of such an inequality appears at first sight 
to be quite unexpected. It will turn out, however, that the 
dimension-dependence of \eqref{eq:pisierc} is not an intrinsic feature of 
the discrete cube, but is simply a reflection of the fact that 
\eqref{eq:pisierc} is not the ``correct'' analogue of the corresponding 
Gaussian inequality. To obtain a dimension-free inequality, we will 
replace $\delta$ by a vector of \emph{biased} Rademacher variables 
$\delta(t)$ which arises naturally in our proof by differentiating the 
discrete heat kernel.

\subsection{A dimension-free Pisier inequality}
\label{sec:mainres}

The following random variables will appear frequently in the sequel, so we 
fix them once and for all. Let $\varepsilon$ be a random vector that is 
uniformly distributed on the cube $\{-1,1\}^n$. Given $t>0$, we let 
$\xi(t)$ be a random vector in the cube, independent of $\varepsilon$, 
whose coordinates $\xi_i(t)$ are independent and identically distributed 
with
$$
	\mathbf{P}\{\xi_i(t)=1\}=\frac{1+e^{-t}}{2},\qquad
	\mathbf{P}\{\xi_i(t)=-1\}=\frac{1-e^{-t}}{2}.
$$
We also define the standardized vector $\delta(t)$ by
$$
	\delta_i(t) := \frac{\xi_i(t)-\mathbf{E}\xi_i(t)}{
	\sqrt{\mathop{\mathrm{Var}}\xi_i(t)}} =
	\frac{\xi_i(t) - e^{-t}}{\sqrt{1-e^{-2t}}}.
$$
The following analogue of \eqref{eq:pisierg} lies at the heart of this 
paper.

\begin{thm}
\label{thm:main}
For any linear space $X$, function
$f:\{-1,1\}^n\to X$, and convex function 
$\Phi:X\to\mathbb{R}$, we have
\begin{equation}
\label{eq:main}
	\mathbf{E}[\Phi(f(\varepsilon)-\mathbf{E}f(\varepsilon))] \le
	\int
	\mathbf{E}\bigg[\Phi\bigg(\frac{\pi}{2}\sum_{j=1}^n
	\delta_j(t) D_jf(\varepsilon)\bigg)\bigg]\,\mu(dt),
\end{equation}
where $\mu$ is the probability measure on $\mathbb{R}_+$ with density
$\mu(dt):=\frac{2}{\pi}\frac{1}{\sqrt{e^{2t}-1}}dt$.
\end{thm}

Even though \eqref{eq:main} is formulated in terms of the biased variables 
$\delta_j(t)$ as opposed to the Rademacher variables $\delta_j$ that 
appear in \eqref{eq:pisierc}, the proof of Theorem \ref{thm:enflo} will 
follow readily by a routine symmetrization argument. For this purpose the 
precise distribution of the random variables $\delta_i(t)$ is in fact 
immaterial: it suffices that they are independent, centered, and have 
bounded variance. However, other applications (such as Theorem 
\ref{thm:cotype} below) do require more precise information on the 
distribution of $\delta_i(t)$, which can be read off from its definition.

\begin{rem}
It is interesting to note that \eqref{eq:main} is not just an analogue 
of \eqref{eq:pisierg} on the cube: it is in fact a strictly stronger 
result, as the Gaussian inequality can be derived 
from Theorem \ref{thm:main} by the central limit theorem. To see why, 
assume 
$f:\mathbb{R}^n\to X$ is a sufficiently smooth function with compact 
support and let $\Phi:X\to\mathbb{R}$ be a sufficiently regular convex 
function. Define $f_N:\{-1,1\}^{n\times N}\to X$ by
$$
	f_N(\varepsilon) := f\bigg(
	\frac{\sum_{j=1}^N\varepsilon_{1j}}{\sqrt{N}},\ldots,
	\frac{\sum_{j=1}^N\varepsilon_{nj}}{\sqrt{N}}
	\bigg),
$$
and note that for $1\le i\le n$ and $1\le j\le N$
$$
	D_{ij}f_N(\varepsilon) =
	\frac{\varepsilon_{ij}}{\sqrt{N}}
	\frac{\partial f}{\partial x_i}
	\bigg(
        \frac{\sum_{j=1}^N\varepsilon_{1j}}{\sqrt{N}},\ldots,
        \frac{\sum_{j=1}^N\varepsilon_{nj}}{\sqrt{N}}
        \bigg) + 
	o\bigg(\frac{1}{\sqrt{N}}\bigg)
$$
as $N\to\infty$. Thus by Theorem \ref{thm:main}
\begin{align*}
	&\mathbf{E}[\Phi(f_N(\varepsilon)-
	\mathbf{E} f_N(\varepsilon))]
	\\ &\le
        \int
        \mathbf{E}\bigg[\Phi\bigg(\frac{\pi}{2}
	\sum_{i=1}^n
	\frac{\sum_{j=1}^N\delta_{ij}(t) \varepsilon_{ij}}{\sqrt{N}}
        \frac{\partial f}{\partial x_i}
        \bigg(
        \frac{\sum_{j=1}^N\varepsilon_{1j}}{\sqrt{N}},\ldots,
        \frac{\sum_{j=1}^N\varepsilon_{nj}}{\sqrt{N}}
        \bigg)\bigg)
	\bigg]\,\mu(dt)
	+o(1).
\end{align*}
Letting $N\to\infty$ now yields \eqref{eq:pisierg} by the multivariate
central limit theorem, as $\{(\varepsilon_{1j},\delta_{1j}(t) 
\varepsilon_{1j},\ldots,\varepsilon_{nj},\delta_{nj}(t)
\varepsilon_{nj})\}_{j\le N}$ are i.i.d.\ random vectors with 
unit covariance matrix.
The requisite regularity assumptions on $f$ and $\Phi$ can subsequently be 
removed by routine approximation arguments.

The above discussion also shows that the constant in Theorem 
\ref{thm:main} is optimal. Indeed, as \eqref{eq:main} implies
\eqref{eq:pisierg}, it suffices
to show that \eqref{eq:pisierg} is sharp. But this is already known
to be the case when $X=\mathbb{R}$ and $\Phi(x)=|x|$
\cite[Chapter 8]{Led96}.
\end{rem}

When $\Phi(x)=\|x\|^p$, the conclusion of Theorem \ref{thm:main} may be 
slightly improved. As the improvement will be needed in the sequel, we 
spell out this variant separately.

\begin{thm}
\label{thm:mainlp}
Let $\mu$ be as in Theorem \ref{thm:main}. Then
for any Banach space $(X,\|\cdot\|)$, function
$f:\{-1,1\}^n\to X$, and $1\le p<\infty$, we have
\begin{equation}
\label{eq:mainlp}
	(\mathbf{E}
	\|f(\varepsilon)-\mathbf{E}f(\varepsilon)\|^p)^{1/p}
	\le
	\frac{\pi}{2}
	\int
	\bigg(\mathbf{E}
	\bigg\|\sum_{j=1}^n
	\delta_j(t) D_jf(\varepsilon)\bigg\|^p\bigg)^{1/p}
	\mu(dt).
\end{equation}
\end{thm}

In this setting, the difference between \eqref{eq:main} and 
\eqref{eq:mainlp} is that in the former the exponent $1/p$ appears
outside the $\mu(dt)$ integral on the right-hand side.

We now briefly describe the idea behind the proofs of Theorems 
\ref{thm:main} and \ref{thm:mainlp}, which was inspired by Gaussian 
semigroup methods of Ledoux \cite[Chapter 8]{Led96}. Instead of using 
rotational invariance as in the proof of \eqref{eq:pisierg} to interpolate 
between $f(G)$ and $f(G')$, we use the heat semigroup on the discrete cube 
to interpolate between $f$ and $\mathbf{E}f$. The resulting expressions 
involve derivatives of the form $D_je^{t\Delta}f$. We now observe that 
rather than applying the derivative to $f$, we may differentiate the heat 
kernel instead. A short computation (Lemma \ref{lem:probrep}) shows that 
the gradient of the heat kernel on the cube yields the biased Rademacher 
vector $\delta(t)$. This elementary observation, analogous to the 
classical smoothing property of diffusion semigroups, leads us to discover 
\eqref{eq:main} in a completely natural manner.

\subsection{Pisier's inequality and cotype}

Theorems \ref{thm:main} and \ref{thm:mainlp} provide dimension-free 
analogues of Pisier's inequality on the cube for an \emph{arbitrary} 
Banach space $X$. With these results in hand, however, we can now revisit 
the question of what additional assumption must be imposed on $X$ in order 
that Pisier's original inequality \eqref{eq:pisierc} holds with a 
dimension-independent constant.

Recall that a Banach space $(X,\|\cdot\|)$ has (Rademacher)
\emph{cotype} $q\in[2,\infty)$ if there exists
$C\in(0,\infty)$ so that for all $n\ge 1$ and 
$x_1,\ldots,x_n\in X$
$$
	\sum_{j=1}^n \|x_j\|^q\le 
	C^q\,\mathbf{E}\Bigg\|\sum_{j=1}^n\varepsilon_jx_j\Bigg\|^q.
$$
We denote by $C_q(X)$ the smallest possible constant $C$ in this 
inequality. The significance of cotype in the present context is 
twofold:
\begin{enumerate}[$\bullet$]
\itemsep\abovedisplayskip
\item If $X$ has finite cotype, one can estimate biased 
Rademacher averages by regular Rademacher averages \cite[Proposition 
3.2]{Pis86}, so that Theorem \ref{thm:mainlp} yields \eqref{eq:pisierc}.
\item If $X$ does not have finite cotype, it contains $\ell_\infty^n$
uniformly \cite[Theorem 3.3]{Pis86}, which enables us to embed Talagrand's
example \cite[section 6]{Tal93} for every $n$.
\end{enumerate}
Both theorems applied here are classical results of Maurey and Pisier. 
These observations give rise to the following characterization.

\begin{thm}
\label{thm:cotype}
For any Banach space $X$ and $1\le p<\infty$, Pisier's 
inequality \eqref{eq:pisierc} holds with a constant independent of 
dimension $n$ if and only if $X$ has finite cotype.
\end{thm}

As any Banach space with nontrivial type has finite cotype \cite[Theorem 
7.1.14]{HNVW17}, we obtain in particular an affirmative answer to the 
question posed after \eqref{eq:pisierc}: Pisier's inequality holds with a 
dimension-free constant in any Banach space with nontrivial type. However, 
one may argue that this fact is no longer of great importance in view of 
our main results; in practice Theorems \ref{thm:main} and \ref{thm:mainlp} 
may be just as easily deployed directly in applications (as we do in 
Theorem \ref{thm:enflo}), and give rise to much better constants than 
would be obtained from Theorem \ref{thm:cotype}.

A quantitative formulation of Theorem \ref{thm:cotype} will be given in 
section \ref{sec:cotype}.

\subsection{Organization of this paper}

The rest of this paper is organized as follows. Section 
\ref{sec:main} is devoted to the proofs of Theorems \ref{thm:main} and 
\ref{thm:mainlp}. We subsequently deduce Theorem \ref{thm:enflo} in 
section \ref{sec:enfloproof}. Finally, Theorem \ref{thm:cotype} is proved 
in section \ref{sec:cotype}.

\section{Proof of Theorems \ref{thm:main} and \ref{thm:mainlp}}
\label{sec:main}

The Laplacian on the discrete cube is defined by
$$
	\Delta f := -\sum_{j=1}^n D_jf.
$$
We denote by $P_t$ the standard heat semigroup on the cube, that is,
$$
	P_t := e^{t\Delta}.
$$
Recall that
$\Delta$ is self-adjoint on $L^2(\{-1,1\}^n)$ with quadratic form
$$
	-\mathbf{E}[f(\varepsilon)\,\Delta g(\varepsilon)] =
	\sum_{j=1}^n \mathbf{E}[D_jf(\varepsilon)\, D_jg(\varepsilon)].
$$
The basis for the proof of Theorem \ref{thm:main} is the following 
probabilistic representation of the heat semigroup and its discrete
partial derivatives.

\begin{lem}
\label{lem:probrep}
We have
$$
	P_tf(x) = \mathbf{E}[f(x_1\xi_1(t),\ldots,x_n\xi_n(t))]
	\quad  \mbox{for }t\ge 0,
$$
and 
$$
	D_jP_t f(x) =
	\frac{1}{\sqrt{e^{2t}-1}}\,
	\mathbf{E}[\delta_j(t)f(x_1\xi_1(t),\ldots,x_n\xi_n(t))]
	\quad \mbox{for } t>0.
$$
\end{lem}

\begin{proof}
Let $Q_tf(x) := \mathbf{E}[f(x_1\xi_1(t),\ldots,x_n\xi_n(t))]$.
By the definition of $\xi_j(t)$, we have
$$
	Q_tf(x) =
	\sum_{\xi\in\{-1,1\}^n}
	\Bigg[
	\prod_{i=1}^n
	\frac{1+e^{-t}\xi_i}{2}\Bigg]
	f(x_1\xi_1,\ldots,x_n\xi_n).
$$
Note also that 
$$
	Q_tf(x_1,\ldots,-x_j,\ldots,x_n) =
	\sum_{\xi\in\{-1,1\}^n}
	\Bigg[
	\prod_{i=1}^n
        \frac{1+e^{-t}\xi_i}{2}
	\Bigg]
	\frac{1-e^{-t}\xi_j}{1+e^{-t}\xi_j}
	f(x_1\xi_1,\ldots,x_n\xi_n).
$$
We now observe that
$$
	\frac{1}{2}\bigg(1-\frac{1-e^{-t}\xi_j}{1+e^{-t}\xi_j}
	\bigg) =
	\frac{e^{-t}\xi_j}{1+e^{-t}\xi_j} =
	\frac{e^{-t}}{1-e^{-2t}}
	(\xi_j-e^{-t}).
$$
We have therefore shown that
$$
	D_jQ_t f(x) =
	\frac{1}{\sqrt{e^{2t}-1}}\,
	\mathbf{E}[\delta_j(t)f(x_1\xi_1(t),\ldots,x_n\xi_n(t))].
$$
It remains to show that $Q_tf=P_tf$. To this end, note that
$Q_0f=f$ and
\begin{align*}
	\frac{d}{dt}Q_tf(x) &=
	-\sum_{j=1}^n
	\sum_{\xi\in\{-1,1\}^n}
	\Bigg[
	\prod_{i=1}^n
	\frac{1+e^{-t}\xi_i}{2}\Bigg]
	\frac{e^{-t}\xi_j}{1+e^{-t}\xi_j}
	f(x_1\xi_1,\ldots,x_n\xi_n)
	\\
	&= -\sum_{j=1}^n D_jQ_tf(x) = \Delta Q_tf(x).
\end{align*}
Thus $Q_t$ satisfies the Kolmogorov equation for the
semigroup $P_t$.
\end{proof}

We are now ready to prove Theorem \ref{thm:main}.

\begin{proof}[Proof of Theorem \ref{thm:main}]
We may assume without loss of generality that $X$ is finite-dimensional
(as $f(\{-1,1\}^n)$ spans a space of dimension at most $2^n$).
Write
$$
	\Phi(x) = \sup_{z\in X^*}
	\{\langle z,x\rangle - \Phi^*(z)\}
$$
where $\Phi^*:X^*\to (-\infty,\infty]$ is the convex conjugate of 
$\Phi$. Then
$$
	\mathbf{E}[\Phi(f(\varepsilon)-\mathbf{E}f(\varepsilon))] =
	\sup_{g:\{-1,1\}^n\to X^*}\{
	\mathbf{E}[\langle g(\varepsilon),f(\varepsilon)-
	\mathbf{E}f(\varepsilon)\rangle] -
	\mathbf{E}[\Phi^*(g(\varepsilon))]\}.
$$
As 
$P_0f=f$ and $\lim_{t\to\infty}P_tf=\mathbf{E}f(\varepsilon)$
(this follows, e.g., from Lemma \ref{lem:probrep}), we can write by the
fundamental theorem of calculus 
\begin{align*}
	\mathbf{E}[\langle g(\varepsilon),f(\varepsilon)-
        \mathbf{E}f(\varepsilon)\rangle]
	&=
	-\int_0^\infty
	\mathbf{E}\bigg[
	\bigg\langle g(\varepsilon),\frac{d}{dt}P_tf(\varepsilon)
	\bigg\rangle\bigg]\,dt 
	\\ &=
	-\int_0^\infty
	\mathbf{E}[\langle g(\varepsilon),\Delta P_tf(\varepsilon)
	\rangle]\,dt \\
	&=
	\int_0^\infty
	\sum_{j=1}^n
	\mathbf{E}[\langle D_j P_t g(\varepsilon),D_j f(\varepsilon)
	\rangle]\,dt,
\end{align*}
where we used in the last line that $\Delta$ is self-adjoint and commutes
with $P_t$. To proceed, we note that by Lemma \ref{lem:probrep}
$$
	\sum_{j=1}^n
	\mathbf{E}[\langle D_j P_t g(\varepsilon),D_j f(\varepsilon)
        \rangle] =
	\frac{1}{\sqrt{e^{2t}-1}}\,
	\mathbf{E}\bigg[
	\bigg\langle g(\varepsilon\xi(t)),
	\sum_{j=1}^n\delta_j(t)D_j f(\varepsilon)
        \bigg\rangle\bigg],
$$
where $\varepsilon\xi(t):=(\varepsilon_1\xi_1(t),\ldots,
\varepsilon_n\xi_n(t))$.
Moreover, $\mathbf{E}[\Phi^*(g(\varepsilon))]=
\mathbf{E}[\Phi^*(g(\varepsilon\xi(t)))]$, as the random vectors 
$\varepsilon\xi(t)$ and $\varepsilon$ have the
same distribution. Thus
\begin{align*}
	&\mathbf{E}[\langle g(\varepsilon),f(\varepsilon)-
        \mathbf{E}f(\varepsilon)\rangle]-
	\mathbf{E}[\Phi^*(g(\varepsilon))] \\
	&=
	\int
        \mathbf{E}\bigg[
        \bigg\langle g(\varepsilon\xi(t)),
	\frac{\pi}{2}
        \sum_{j=1}^n\delta_j(t)D_j f(\varepsilon)
        \bigg\rangle-
	\Phi^*(g(\varepsilon\xi(t)))
	\bigg]
	\,\mu(dt) \\
	&\le
	\int
        \mathbf{E}\bigg[
	\Phi\bigg(
	\frac{\pi}{2}
        \sum_{j=1}^n\delta_j(t)D_j f(\varepsilon)
        \bigg)
	\bigg]
	\,\mu(dt),
\end{align*}
and the conclusion follows.
\end{proof}

The proof of Theorem \ref{thm:mainlp} is almost identical.

\begin{proof}[Proof of Theorem \ref{thm:mainlp}]
In this case we use \cite[Proposition 1.3.1]{HNVW16}
$$
	(\mathbf{E}\|f(\varepsilon)-\mathbf{E}f(\varepsilon)\|^p)^{1/p}
	=
	\sup_{\mathbf{E}\|g(\varepsilon)\|^q\le 1}
	\mathbf{E}[\langle g(\varepsilon),
	f(\varepsilon)-\mathbf{E}f(\varepsilon)\rangle]
$$
with $\frac{1}{p}+\frac{1}{q}=1$. Proceeding exactly as in the proof of
Theorem \ref{thm:main}, we obtain
\begin{align}
\label{eq:dualdual}
	\mathbf{E}[\langle g(\varepsilon),f(\varepsilon)-
        \mathbf{E}f(\varepsilon)\rangle] 
	&=
	\int
        \mathbf{E}\bigg[
        \bigg\langle g(\varepsilon\xi(t)),
	\frac{\pi}{2}
        \sum_{j=1}^n\delta_j(t)D_j f(\varepsilon)
        \bigg\rangle
	\bigg]
	\,\mu(dt) \\
	&\nonumber\le
	\int
 	(\mathbf{E}\|g(\varepsilon\xi(t))\|^q)^{1/q}
        \bigg(\mathbf{E}
        \bigg\|\frac{\pi}{2}\sum_{j=1}^n
        \delta_j(t) D_jf(\varepsilon)\bigg\|^p\bigg)^{1/p}
	\,\mu(dt)
\end{align}
using H\"older's inequality.
Recalling that $\mathbf{E}\|g(\varepsilon\xi(t))\|^q=
\mathbf{E}\|g(\varepsilon)\|^q$ as the random vectors $\varepsilon\xi(t)$ 
and $\varepsilon$ have the same distribution, the conclusion 
follows readily.
\end{proof}

\begin{rem}[Alternative approach to the proofs of Theorems \ref{thm:main} 
and \ref{thm:mainlp}]
Using that $\varepsilon$ and $\varepsilon\xi(t)$
have the same distribution and that \eqref{eq:dualdual} holds for 
all $g$, it is readily seen that \eqref{eq:dualdual} implies the 
\emph{pointwise} identity
\begin{equation}
\label{eq:pointwise}
	f(x)-\mathbf{E}f(\varepsilon) =
	\int 
	\mathbf{E}\Bigg[\frac{\pi}{2} \sum_{j=1}^n 
	\delta_j(t)D_jf(x\xi(t))\Bigg]
	\,\mu(dt)
\end{equation}
for $x\in\{-1,1\}^n$. By using this identity one can organize the proofs
in a manner that is closer to the proof of \eqref{eq:pisierg}. For 
example, to prove Theorem \ref{thm:main} we can upper 
bound $\Phi(f(x)-\mathbf{E}f(\varepsilon))$ pointwise by applying Jensen's
inequality to the right-hand side of \eqref{eq:pointwise}, and then
\eqref{eq:main} follows by taking the expectation of the resulting
expression and using that $\varepsilon$ and $\varepsilon\xi(t)$ have the
same distribution.

The pointwise identity \eqref{eq:pointwise} can also be proved directly, 
which leads to proofs of Theorems \ref{thm:main} and \ref{thm:mainlp} that 
avoid the use of duality. The following argument was 
communicated to us by Jingbo Liu. First, note two basic properties of the 
discrete cube: $D_j^2=D_j$ and $D_j P_t=P_t D_j$ for every $j$. Thus we 
can write
\begin{align*}
	f(x)-\mathbf{E}f(\varepsilon) &=
	-\int_0^\infty \frac{d}{dt}P_tf(x)\,dt
	=
	-\int_0^\infty \Delta P_tf(x)\,dt
	\\ &=
	\int_0^\infty \sum_{j=1}^n D_j^2 P_tf(x)\,dt
	=
	\int_0^\infty \sum_{j=1}^n D_j P_t D_jf(x)\,dt.
	\\
	&=
	\int_0^\infty
	\frac{1}{\sqrt{e^{2t}-1}}\,
	\mathbf{E}\Bigg[\sum_{j=1}^n 
	\delta_j(t)D_jf(x\xi(t))\Bigg]
	\,dt,
\end{align*}
using Lemma \ref{lem:probrep} in the last step.
While conceptually appealing,
the disadvantage of this argument is that it relies on 
special properties of calculus on the discrete cube. In contrast, the 
proofs that are based on duality use nothing else than the quadratic form 
$-\langle f,\Delta g\rangle_{\ell^2(\{-1,1\})}=\langle 
Df,Df\rangle_{\ell^2(\{-1,1\})}$ and the gradient formula of Lemma 
\ref{lem:probrep}, providing a more direct route to extensions beyond 
the discrete cube.
\end{rem}

\section{Proof of Theorem \ref{thm:enflo}}
\label{sec:enfloproof}

Theorem \ref{thm:enflo} follows from Theorem 
\ref{thm:main} by a routine symmetrization argument.

\begin{proof}[Proof of Theorem \ref{thm:enflo}]
The first inequality $T_p^\mathrm{R}(X) \le T_p^\mathrm{E}(X)$ 
follows 
readily by choosing $f(\varepsilon)=\sum_{j=1}^n \varepsilon_jx_j$ in the 
definition of Enflo type.

In the converse direction, note first that as $\varepsilon$ and 
$-\varepsilon$ have the same distribution, and as $x\mapsto\|x\|^p$ is 
convex, we can estimate
$$
	\mathbf{E}\bigg\|\frac{f(\varepsilon)-f(-\varepsilon)}{2}\bigg\|^p
	=
	\mathbf{E}\bigg\|\frac{f(\varepsilon)-\mathbf{E}f(\varepsilon)
	-f(-\varepsilon)+\mathbf{E}f(-\varepsilon)}{2}\bigg\|^p
	\le
	\mathbf{E}\|f(\varepsilon)-\mathbf{E}f(\varepsilon)\|^p.
$$
Applying Theorem \ref{thm:main} 
with $\Phi(x)=\|x\|^p$ yields
$$
	\mathbf{E}\bigg\|\frac{f(\varepsilon)-f(-\varepsilon)}{2}\bigg\|^p
	\le
	\int
	\mathbf{E}\Bigg\|\frac{\pi}{2}\sum_{j=1}^n
	\delta_j(t) D_jf(\varepsilon)\Bigg\|^p\mu(dt).
$$
To estimate the right-hand side we use a standard symmetrization argument.
Let $\xi'(t)$ be an independent copy of $\xi(t)$ and 
$\varepsilon'$ be an independent copy of $\varepsilon$. Then
\begin{align*}
	\mathbf{E}\Bigg\|\sum_{j=1}^n
	\delta_j(t) D_jf(\varepsilon)\Bigg\|^p
	&\le
	\mathbf{E}\Bigg\|\sum_{j=1}^n
	\frac{\xi_j(t)-\xi_j'(t)}{\sqrt{\mathop\mathrm{Var}\xi_j(t)}}
	D_jf(\varepsilon)\Bigg\|^p
	\\
	&=
	\mathbf{E}\Bigg\|\sum_{j=1}^n
	\varepsilon_j'
	\frac{\xi_j(t)-\xi_j'(t)}{\sqrt{\mathop\mathrm{Var}\xi_i(t)}}
	D_jf(\varepsilon)\Bigg\|^p
	\\
	&\le
	T_p^{\mathrm{R}}(X)^p
	\sum_{j=1}^n
	\mathbf{E}
	\bigg|\frac{\xi_j(t)-\xi_j'(t)}{\sqrt{\mathop\mathrm{Var}\xi_i(t)}}
	\bigg|^p
	\mathbf{E}\|D_jf(\varepsilon)\|^p,
\end{align*}
where we used Jensen's inequality in the first line; that 
$\xi(t)-\xi'(t)$ has the same distribution as 
$\varepsilon'(\xi(t)-\xi'(t))$ (by symmetry and independence) in the 
second line; and the definition of Rademacher type 
conditionally on $\xi(t),\xi'(t),\varepsilon$
and that $\xi(t),\xi'(t),\varepsilon,\varepsilon'$ are independent
in the third line. But as $p\le 2$, we obtain
$$
        \mathbf{E}
        \bigg|\frac{\xi_j(t)-\xi_j'(t)}{\sqrt{\mathop\mathrm{Var}\xi_i(t)}}
        \bigg|^p
	\le
        \bigg(
        \frac{\mathbf{E}[(\xi_j(t)-\xi_j'(t))^2]}{\mathop\mathrm{Var}\xi_i(t)}
        \bigg)^{p/2}
	=
	2^{p/2}
$$
by Jensen's inequality. Thus we have shown
$$
        \mathbf{E}\bigg\|\frac{f(\varepsilon)-f(-\varepsilon)}{2}\bigg\|^p
	\le
	\bigg(
	\frac{\pi}{\sqrt{2}}
	T_p^{\mathrm{R}}(X)
	\bigg)^p
        \sum_{j=1}^n
        \mathbf{E}\|D_jf(\varepsilon)\|^p,
$$
which implies $T_p^{\mathrm{E}}(X) \le \frac{\pi}{\sqrt{2}}
T_p^{\mathrm{R}}(X)$.
\end{proof}

\section{Proof of Theorem \ref{thm:cotype}}
\label{sec:cotype}

The following contraction principle is a classical result of Maurey and 
Pisier (see, e.g., \cite[Proposition 3.2]{Pis86}). We spell out a version 
with explicit constants.

\begin{thm}
\label{thm:mp}
Let $(X,\|\cdot\|)$ be a Banach space of cotype $q<\infty$,
let $\eta_1,\ldots,\eta_n$ be i.i.d.\ symmetric random variables,
and let $\varepsilon$ be uniformly distributed on $\{-1,1\}^n$. Then
for any $n\ge 1$, $x_1,\ldots,x_n\in X$, and $1\le p<\infty$, we have
$$
	\Bigg(
	\mathbf{E}\Bigg\|\sum_{j=1}^n \eta_j x_j\Bigg\|^p\Bigg)^{1/p}
	\le
	L_{q,p}
	\int_0^\infty \mathbf{P}\{|\eta_1|>t\}^{\frac{1}{\max(q,p)}}dt\,
	\Bigg(
	\mathbf{E}\Bigg\|\sum_{j=1}^n \varepsilon_j x_j\Bigg\|^p
	\Bigg)^{1/p}
$$
with $L_{q,p}=L\,C_q(X)\max(1,(q/p)^{1/2})$,
where $L$ is a universal constant.
\end{thm}

\begin{proof}
As $\eta_i$ are symmetric random variables, they have the same 
distribution as $\varepsilon_i\eta_i$. The conclusion for the special case
$p=q$ follows from \cite[Theorem 7.2.6]{HNVW17}. For the general case, we 
consider two distinct cases.

For the case $p>q$, recall that a Banach space with cotype $q$ also has 
cotype $r$ for all $r>q$, with $C_r(X)\le C_q(X)$ \cite[p.\,55]{HNVW17}.
Thus the conclusion follows readily from \cite[Theorem 7.2.6]{HNVW17} by 
choosing $q=p$.

For the case $p<q$, we bound the $L^p$-norm on the left-hand side by the 
$L^q$-norm, and then apply the inequality for the case $p=q$. This yields
$$
	\Bigg(
	\mathbf{E}\Bigg\|\sum_{j=1}^n \eta_j x_j\Bigg\|^p\Bigg)^{1/p}
	\le
	L\,C_q(X)
	\int_0^\infty \mathbf{P}\{|\eta_1|>t\}^{\frac{1}{q}}dt\,
	\Bigg(
	\mathbf{E}\Bigg\|\sum_{j=1}^n \varepsilon_j x_j\Bigg\|^q
	\Bigg)^{1/q}.
$$
We conclude by using the Kahane-Khintchine inequality \cite[Theorem 
6.2.4]{HNVW17} to bound the 
$L^q$-norm on the right-hand side by the $L^p$-norm, which incurs
the additional factor $\lesssim (q/p)^{1/2}$. This completes the proof.
\end{proof}

We are now ready to prove one direction of Theorem \ref{thm:cotype}: if 
$X$ has finite cotype, then \eqref{eq:pisierc} holds with a dimension-free 
constant.

\begin{prop}
\label{prop:cotype}
Let $(X,\|\cdot\|)$ be a Banach space of cotype $q$, and let
$\varepsilon,\delta$ be independent uniformly distributed random vectors
in $\{-1,1\}^n$. Then for any
function $f:\{-1,1\}^n\to X$ and $1\le p<\infty$, we have
$$
        \mathbf{E}\|f(\varepsilon)-\mathbf{E}f(\varepsilon)\|^p
        \le
        K_{q,p}^p\,
        \mathbf{E}\Bigg\|\sum_{j=1}^n \delta_j D_jf(\varepsilon)
        \Bigg\|^p
$$
with $K_{q,p}=K\, C_q(X)\, p \max(1,(q/p)^{3/2})$,
where $K$ is a universal constant.
\end{prop}

\begin{proof}
Let $\xi'(t)$ be an independent copy of $\xi(t)$.
We first note that
$$
	\int_0^\infty \mathbf{P}\{|\xi_j(t)-\xi_j'(t)|>s\}^{1/r}ds
	=
	2^{1-1/r} (1-e^{-2t})^{1/r}.
$$
Thus
\begin{align*}
        \bigg(\mathbf{E}
        \bigg\|\sum_{j=1}^n
        \delta_j(t) D_jf(\varepsilon)\bigg\|^p\bigg)^{1/p} 
	&\le
	\frac{1}{\sqrt{1-e^{-2t}}}
        \bigg(\mathbf{E}
        \bigg\|\sum_{j=1}^n
        (\xi_j(t)-\xi_j'(t)) D_jf(\varepsilon)\bigg\|^p\bigg)^{1/p}
	\\
	&\le
	2L_{q,p} 
	(1-e^{-2t})^{\frac{1}{\max(q,p)}-\frac{1}{2}}
        \bigg(\mathbf{E}
        \bigg\|\sum_{j=1}^n
        \delta_j D_jf(\varepsilon)\bigg\|^p\bigg)^{1/p},
\end{align*}
where we used Jensen's inequality in the first line and we applied
Theorem \ref{thm:mp} conditionally on $\varepsilon$ in the second line.
Now note that 
\begin{align*}
	\frac{\pi}{2}
	\int 
	(1-e^{-2t})^{\frac{1}{\max(q,p)}-\frac{1}{2}}
	\mu(dt) 
	\le
	\int_0^\infty
	e^{-t}
	(1-e^{-t})^{\frac{1}{\max(q,p)}-1}\,dt
	= \max(q,p).
\end{align*}
Thus Theorem \ref{thm:mainlp} yields
$$
        (\mathbf{E}\|f(\varepsilon)-\mathbf{E}f(\varepsilon)\|^p)^{1/p}
	\le
	2L_{q,p}\max(q,p)
	\Bigg(\mathbf{E}\Bigg\|\sum_{j=1}^n \delta_j D_jf(\varepsilon)
        \Bigg\|^p\Bigg)^{1/p},
$$
and the conclusion follows by the definition of $L_{q,p}$ in
Theorem \ref{thm:mp}.
\end{proof}

\begin{rem}
It should be noted that the improvement provided by Theorem 
\ref{thm:mainlp} is used crucially in the proof of Theorem 
\ref{thm:cotype}. Had we used Theorem \ref{thm:main} instead, we would 
have encountered the integral
$\int_0^\infty 
(1-e^{-2t})^{\frac{p}{\max(q,p)}-\frac{p}{2}}\mu(dt)$ in the proof; 
it is readily verified that this integral
diverges at some finite value of $p$.
\end{rem}

It remains to show the converse direction: if $X$ does not have finite 
cotype, then the constant in \eqref{eq:pisierc} is at least of order $\log 
n$.

\begin{prop}
\label{prop:cotypeno}
If the Banach space $X$ does not have finite cotype, then for every $n\ge 
1$ and $1\le p<\infty$, there exists a function $f:\{-1,1\}^n\to X$ so 
that
$$
	\mathbf{E}\|f(\varepsilon)-\mathbf{E}f(\varepsilon)\|^p \ge
	C_{n,p}^p \,
	\mathbf{E}\Bigg\|\sum_{j=1}^n \delta_j D_jf(\varepsilon)
        \Bigg\|^p
$$
with $C_{n,p} = C\log(n/9p)$, where $C$ is a universal constant.
\end{prop}

\begin{proof}
It was shown by Talagrand \cite[section 6]{Tal93} that for every $n\ge 1$
and $1\le p<\infty$, there is a function 
$f:\{-1,1\}^n\to\ell_\infty^{2^n}$ so that
$$
	(\mathbf{E}\|f(\varepsilon)-\mathbf{E}f(\varepsilon)\|_\infty^p)^{1/p} 
	\ge
	c\log(n/9p)
	\Bigg(
	\mathbf{E}\Bigg\|\sum_{j=1}^n \delta_j D_jf(\varepsilon)
        \Bigg\|_\infty^p
	\Bigg)^{1/p}
$$
for a universal constant $c$. But if $X$ does not have finite cotype, 
then by the Maurey-Pisier theorem \cite[Theorem 7.3.8]{HNVW17} it must
contain a $2$-isomorphic copy of $\ell_\infty^N$ for every $N\ge 1$. Thus
we can embed Talagrand's example in $X$ for every $n\ge 1$ 
and $1\le p<\infty$, and the proof is readily concluded.
\end{proof}

\begin{rem}
We emphasize that our characterization of when Pisier's inequality holds
with dimension-free constant assumes the Banach space $X$ and
$1\le p<\infty$ are fixed. When this is not the case, other phenomena can 
arise. For example, it follows from a result of Wagner \cite{Wag00} that 
if one chooses $p\asymp n$, then \eqref{eq:pisierc} holds with a universal 
constant for \emph{any} Banach space $X$. This is is a purely 
combinatorial fact that does not capture any structure of the underlying 
space.
\end{rem}

\subsection*{Acknowledgments}

The authors thank Sergey Bobkov, Alexandros Eskenazis, Dong Li, Jingbo 
Liu, and Gilles Pisier for helpful discussions and comments. We are 
particularly grateful to Assaf Naor and to the anonymous referee for 
suggestions that considerably improved the presentation of the paper.

P.I.\ was supported in part by NSF grants DMS-1856486 and 
CAREER-DMS-1945102. R.v.H.\ was supported in part by NSF grants 
CAREER-DMS-1148711 and DMS-1811735, by the ARO through PECASE award 
W911NF-14-1-0094, and by the Simons Collaboration on Algorithms \& 
Geometry. A.V.\ was supported in part by NSF grants DMS-160065 and 
DMS-1900268.


\end{document}